\documentclass{conm-p-l}

\usepackage{amssymb,amsmath,relsize}

\usepackage{color}

\newtheorem{thm}{Theorem}[section]
\newtheorem{lem}[thm]{Lemma}

\newtheorem{prop}[thm]{Proposition}

\theoremstyle{definition}
\newtheorem{defn}[thm]{Definition}
\newtheorem{exmpl}[thm]{Example}

\theoremstyle{remark}
\newtheorem{rem}[thm]{Remark}

\numberwithin{equation}{section}

\begin{document}

\title{Invariant for Sets of Pfister Forms}


\author{Adam Chapman}
\address{School of Computer Science, Academic College of Tel-Aviv-Yaffo, Rabenu Yeruham St., P.O.B 8401 Yaffo, 6818211, Israel}
\email{adam1chapman@yahoo.com}

\author{Ilan Levin}
\address{Department of Mathematics, Bar-Ilan University, Ramat Gan, 55200}
\email{ilan7362@gmail.com}

\subjclass[2020]{Primary 19D45;
Secondary 11E04, 11E81, 13A35, 16K20 20G10}
\keywords{Algebraic $K$-Theory, Milnor $K$-Theory, Symmetric Bilinear Forms, Quadratic Forms, Symbol Length, Quaternion Algebras}


\begin{abstract}
We associate an $(n_1+\dots+n_t-k(t-1))$-fold Pfister form to any $t$-tuple of $k$-linked Pfister forms of dimensions $2^{n_1},\dots,2^{n_t}$, and prove its invariance under the different symbol presentations of the forms with a common $k$-fold sub-symbol. We then show that it vanishes when the forms are actually $(k+1)$-linked or when the characteristic is 2 and the forms are inseparably $k$-linked. We study whether the converse statements hold or not.
\end{abstract}

\maketitle


\section{Introduction}

Let $F$ be a field. When $char(F) \neq 2$, an $n$-fold form is a quadratic form of the shape
$$\langle \! \langle \alpha_1,\dots,\alpha_n \rangle \! \rangle = {\mathlarger\perp}_{c_1,\dots,c_n \in \{0,1\}} \langle (-1)^{c_1+\dots+c_n} \alpha_1^{c_1} \cdot \ldots \cdot \alpha_n^{c_n}\rangle,$$
and when $\operatorname{char}(F)=2$, of the shape
$$\langle \! \langle \alpha_1,\dots,\alpha_{n-1},\alpha_n ] \! ] = {\mathlarger\perp}_{c_1,\dots,c_{n-1} \in \{0,1\}} \alpha_1^{c_1} \cdot \ldots \cdot \alpha_{n-1}^{c_{n-1}}\cdot [1,\alpha_n],$$
$$\text{where} \ [1,\alpha]=x^2+xy+\alpha y^2.$$

Pfister forms correspond to the generators of the Galois cohomology groups $H^n(F,\mu_2)$ in characteristic not 2 or the Kato-Milne cohomology groups $H_2^n(F)$ in characteristic 2 (see \cite{EKM}, \cite{Voevodsky} and \cite{Kato:1982}).

When $\operatorname{char}(F)\neq 2$, there is a well-defined map associating to any sum of $m$ symbols in $H^n(F,\mu_2)$, a single symbol in $H^{mn}(F,\mu_2)$ (see \cite[Appendix A]{Kahn:2000}), which is literally the tensor product of the symbols in the sum (as bilinear forms). There is no characteristic 2 analogue to this. The invariant we propose here can serve as a weak characteristic-free analogue, which associates an invariant to sets of Pfister forms (instead of sums of symbols in the cohomology groups), under the condition they share a certain common quadratic Pfister factor.

Pfister forms $\psi_1 , \dots, \psi_m$ are called \textbf{inseparably} $k$-linked if there exists a $k$-fold \textbf{bilinear} Pfister form $B$ which is a Pfister factor of all the forms $\psi_1,\dots,\psi_m$, i.e., there exist quadratic Pfister forms $\pi_1,\dots, \pi_m$ such that $\psi_1 = B \otimes \pi_1,\dots,\psi_m = B \otimes \pi_m$.
Similarly, Pfister forms $\psi_1,\dots,\psi_m$ are called \textbf{separably} $k$-linked if there exists a $k$-fold \textbf{quadratic} Pfister form $\pi$ which is a Pfister factor of all the forms $\psi_1,\dots,\psi_m$, i.e., there exist bilinear Pfister forms $B_1,\dots, B_m$ such that $\psi_1 = B_1 \otimes \pi,\dots,\psi_m = B_m \otimes \pi$. \\

When the forms are over a field of characteristic different from 2, the two notions coincide. The case of characteristic 2 is more subtle. By \cite[Theorem 4.7]{ChapmanFlorenceMcKinnie:2022}, any finite system of inseparably linked quaternion algebras are also separably linked, and thus every finite number of quadratic Pfister forms that are inseparably $k$-linked are also separably $k$-linked. The converse is false (see \cite{Lam:2002} for example).

The study of linkage proved to be useful in the study of the $u$-invariant and other related arithmetic invariants: when every two 2-fold quadratic Pfister forms are 1-linked, the $u$-invariant is either 0,1,2,4 or 8 (\cite{ElmanLam:1973} and \cite{ChapmanDolphin:2017}), and having $u(F)\leq 4$ is equivalent to 1-linkage of every 3 quadratic 2-fold Pfister forms when $\operatorname{char}(F)\neq 2$ (\cite{ChapmanDolphinLeep}). When $\operatorname{char}(F)=2$, $u(F)\leq 4$ if and only if every 2 quadratic 2-fold Pfister forms are inseparably 1-linked (\cite{Baeza:1982}). For the connection between linkage of pairs of quadratic $n$-fold Pfister forms for $n \geq 3$ and other similar arithmetic invariants, see \cite{ChapmanMcKinnie:2019}.

The term ``$k$-linkage" without specifying whether it is separable or inseparable can be by default considered separable.

In this paper we associate an $(n_1+\dots+n_t-k(t-1))$-fold Pfister form to any $t$-tuple of $k$-linked Pfister forms of dimensions $2^{n_1},\dots,2^{n_t}$, and prove its invariance under the different symbol presentations of the forms with a common $k$-fold sub-symbol. As applications to the study of linkage, we show that it vanishes when the forms are actually $(k+1)$-linked or when the characteristic is 2 and the forms are inseparably $k$-linked. We study whether the converse statements hold or not.

\section{Conditions for $k$-linkage}

\begin{prop}\label{inslinkedcond}
Let $F$ be a field of characteristic 2. Given separably $k$-linked Pfister forms $\varphi$ and $\psi$, they are also inseparably $k$-linked $\iff$ $\omega$ is isotropic, where $ \omega = \theta_{\varphi} \perp \theta_{\psi} \perp \langle 1 \rangle$, with $\theta_{\varphi}$ and $\theta_{\psi}$ such that $\varphi = \pi \perp \theta_{\varphi}$ and $\psi = \pi \perp \theta_{\psi}$, where $\pi$ is a $k$-fold quadratic Pfister form.
\end{prop}

\begin{proof}
Write $\varphi = \langle \! \langle \beta_{1}, \dots, \beta_{n}, \alpha_1, \dots, \alpha_{k}]\!]$ and $\psi = \langle \! \langle \gamma_{1}, \dots, \gamma_{m}, \alpha_1, \dots, \alpha_{k}]\!]$. Denote $\pi = \langle \! \langle \alpha_1 , \dots , \alpha_{k}]\!]$. By definition of a pure Pfister subform, $$\varphi' = \langle \! \langle \beta_{1}, \dots, \beta_{n}, \alpha_1, \dots, \alpha_{k-1} \rangle \! \rangle ' \otimes [1,\alpha_{k}] \perp \langle 1 \rangle = \langle \! \langle \alpha_1, \dots, \alpha_{k-1} \rangle \! \rangle' \otimes [1,\alpha_{k}] \perp \theta_{\varphi} \perp \langle 1 \rangle$$ and $$\psi' = \langle \! \langle \gamma_{1}, \dots, \gamma_{m}, \alpha_1, \dots, \alpha_{k-1} \rangle \! \rangle ' \otimes [1,\alpha_{k}] \perp \langle 1 \rangle = \langle \! \langle \alpha_1, \dots, \alpha_{k-1} \rangle \! \rangle' \otimes [1,\alpha_{k}] \perp \theta_{\psi} \perp \langle 1 \rangle.$$
We know that $\langle 1,1 \rangle \cong \langle 1,0 \rangle$ (since $x^2+y^2 = (x+y)^2$), hence $\varphi' \perp \psi' = (2^{k}-2) \times \mathbb{H} \perp \theta_{\varphi} \perp \theta_{\psi} \perp \langle 1 \rangle = (2^{k}-2) \times \mathbb{H} \perp \omega$. Now, $\omega$ is isotropic $\iff$ $i_{W}(\varphi' \perp \psi') \geq 2^{k}-1$, and by \cite[Theorem 2.5.5]{Faivre:thesis}, $i_{W}(\varphi' \perp \psi') \geq 2^{k}-1$ $\iff$ $\varphi$ and $\psi$ are inseparably $k$-linked. This completes the proof.
\end{proof}
\begin{prop}\label{seplinkedcond}
Let $F$ be a field of arbitrary characteristic. Given $k$-linked Pfister forms $\varphi$ and $\psi$, they are ($k+1$)-linked $\iff$ $\omega$ is isotropic, where $\varphi \perp -\psi = 2^{k} \times \mathbb{H} \perp \omega$
\end{prop}

\begin{proof}
First, write $\varphi = \pi \perp \theta_{\varphi}$ and $\psi = \pi \perp \theta_{\psi}$. Then, $\varphi \perp -\psi = \pi \perp -\pi \perp \theta_{\varphi} \perp -\theta_{\psi}$. So, $\varphi \perp -\psi = 2^{k} \times \mathbb{H} \perp \omega$, where $\omega = \theta_{\varphi} \perp -\theta_{\psi}$. 

($\Longrightarrow$): If $\varphi$ and $\psi$ are ($k+1$)-linked, then $i_W(\varphi \perp -\psi)\geq 2^{k+1}$, and therefore $i_W(\omega) \geq 2^k$, and in particular, $\omega$ is isotropic.

($\Longleftarrow$): Suppose $\omega = \theta_{\varphi} \perp -\theta_{\psi}$ is isotropic. Then $\theta_{\varphi}$ and $-\theta_{\psi}$ represent a common element $a \in F^{\times}$. Let us claim that $\ \langle \! \langle -a \rangle \! \rangle \otimes \pi$ is a Pfister factor of both $\varphi$ and $\psi$. By \cite[Lemma 23.1]{EKM}, $\ \langle \! \langle -a \rangle \! \rangle \otimes \pi$ is a subform of $\varphi$ and $\psi$, and so the forms are $(k+1)$-linked.
\end{proof}

\section{The Invariant}


Fix a field $F$ in the background, and let $PF$ denote the set of quadratic Pfister forms over that field.

\begin{lem}\label{chainstep}
Given integers $n > k \geq 1$, a quadratic $n$-fold Pfister form $\psi$ and two quadratic $k$-fold factors $\varphi_1$ and $\varphi_2$, there exists a bilinear Pfister form $\rho$ for which $\psi=\varphi_1 \otimes \rho$ and $\psi=\varphi_2 \otimes \rho$.
\end{lem}

\begin{proof}
It is enough to show that there exists $\gamma \in F^\times$ for which both $\varphi_1 \otimes \langle \! \langle \gamma \rangle \! \rangle$ and $\varphi_2 \otimes \langle \! \langle \gamma \rangle \! \rangle$ are subforms of $\psi$. Then the process can be repeated for $\varphi_1 \otimes \langle \! \langle \gamma \rangle \! \rangle$ and $\varphi_2 \otimes \langle \! \langle \gamma \rangle \! \rangle$, and so on, as long as the two forms we obtain are of smaller dimension than $\dim \psi$. The process terminates when the dimension of $\varphi_1 \otimes \langle \! \langle \gamma \rangle \! \rangle$ and $\varphi_2 \otimes \langle \! \langle \gamma \rangle \! \rangle$ is the same as $\dim \psi$, and then all three forms are isometeric to each other.

Recall that $\psi$ is endowed with an underlying symmetric bilinear form $B_\psi(v,w)=\psi(v+w)-\psi(v)-\psi(w)$ for any $v,w \in F^{2^n}$, which gives rise to the notation of orthogonality. This symmetric bilinear form is non-degenerate in the sense that there is no vector in $F^{2^n}$ that is perpendicular to all the vectors in $F^{2^n}$.

Now, consider the subspaces $V_1$ and $V_2$ of $F^{2^n}$ corresponding to the subforms $\varphi_1$ and $\varphi_2$ of $\psi$.
Note that $V_1$ and $V_2$ share at least the vector $(1,0,\dots,0)$ whose image under $\psi$ is $1$.
Write $V_1^\perp$ and $V_2^\perp$ for the orthogonal complements of $V_1$ and $V_2$ in $F^{2^n}$.
Their dimension is $2^n-2^k$.
Since they are both subspaces of $W=\operatorname{Span}_F\{(1,0,\dots,0)\}^\perp$ and $\dim V_1^\perp+\dim V_2^\perp>\dim W$, the intersection $V_1^\perp \cap V_2^\perp$ is nontrivial.
Take $v \in V_1^\perp \cap V_2^\perp \setminus \{\vec{0}\}$ and set $\gamma=-\psi(v)$. Then, by \cite[Proposition 6.15 and 15.7]{EKM}, both $\varphi_1 \otimes \langle \! \langle \gamma \rangle \! \rangle$ and $\varphi_2 \otimes \langle \! \langle \gamma \rangle \! \rangle$ are subforms of $\psi$.
\end{proof}

\begin{defn}
Define $PF^{2}(k) = \{ (\psi_1,\psi_2) \in PF \times PF : \psi_1 \text{and } \psi_2 \text{ are $k$-linked} \}$.
\end{defn}

\begin{defn}[{The Invariant}] Define $\mathcal{F}^{2}_{k} : PF^{2}(k) \rightarrow PF$ by: $$(\psi_1 , \psi_2) \mapsto \varphi_{1} \otimes \varphi_{2} \otimes \pi$$
where $\pi$ is a $k$-fold Pfister form for which $\psi_1=\varphi_1 \otimes \pi$ and $\psi_2=\varphi_2 \otimes \pi$ for some bilinear Pfister forms $\varphi_1$ and $\varphi_2$.
\end{defn}

\begin{thm}
$\mathcal{F}^{2}_{k}$ is well-defined.
\end{thm}

\begin{proof}
Let us use colors to make it easier to follow the proof.
Suppose $\psi_1 = \varphi_1 \otimes \pi = \tau_1 \otimes \mu$ and $ \psi_2 = \varphi_2 \otimes \pi = \tau_2 \otimes \mu$ , where $\pi$ and $\mu$ are $k$-Pfister forms. We want to show that $\mathcal{F}^{2}_{k}(\psi_1, \psi_2) = \varphi_1 \otimes \varphi_2 \otimes \pi = \tau_1 \otimes \tau_2 \otimes \mu$.
Since $\pi$ and $\mu$ are sub-Pfister forms of $\psi_1$ and $\psi_2$, and thus by Lemma \ref{chainstep} there exist $\rho_1$ and $\rho_2$ such that \textcolor{red}{$\varphi_1 \otimes \pi = \psi_1 = \rho_1 \otimes \pi = \rho_1 \otimes \mu = \psi_1 = \tau_1 \otimes \mu$} and \textcolor{blue}{$\varphi_2 \otimes \pi = \psi_2 = \rho_2 \otimes \pi = \rho_2 \otimes \mu = \psi_2 = \tau_2 \otimes \mu$}.

And now: $\textcolor{red}{\varphi_1} \otimes \varphi_2 \otimes \textcolor{red}{\pi} = \textcolor{red}{\rho_1} \otimes \varphi_2 \otimes \textcolor{red}{\pi} = \rho_1 \otimes \textcolor{blue}{\varphi_2} \otimes \textcolor{blue}{\pi} = \rho_1 \otimes \textcolor{blue}{\tau_2} \otimes \textcolor{blue}{\mu} = \textcolor{red}{\rho_1} \otimes \tau_2 \otimes \textcolor{red}{\mu} = \textcolor{red}{\tau_1} \otimes \tau_2 \otimes \textcolor{red}{\mu}$, as desired.

\end{proof}

We can now define the invariant for finite sequences of forms inductively. Denote $PF^{m}(k) = \{(\psi_1 , \dots , \psi_m) \in \underbrace{PF \times \dots \times PF}_{m} : \text{all } \psi_{j} \text{ are $k$-linked} \}$. Now, let $$\mathcal{F}^{m}_{k} : PF^{m}(k) \to PF$$ defined by: $\mathcal{F}^{m}_{k}(\psi_1, \dots \psi_m ) = \mathcal{F}^{2}_{k}(\mathcal{F}^{m-1}_{k}(\psi_1, \dots, \psi_{m-1}), \psi_m )$. It is still invariant, since $\mathcal{F}^{m-1}_{k}(\psi_1, \dots, \psi_{m-1})$ is a Pfister form that is $k$-linked to $\psi_m$.
When the number $m$ is obvious from the context, we shall write simply $\mathcal{F}_{k}$.

\begin{rem}
In the special case where:
\begin{enumerate}
\item $\sqrt{-1} \in F$,
\item $\psi_1 , \dots, \psi_{t}$ are $n$-fold, and
\item they are separably ($n-1$)-linked,
\end{enumerate}
we have $\Sigma_{\psi_1 , \dots , \psi_{t}} = \mathcal{F}_{n-1}(\psi_1, \dots, \psi_{t})$,
where $\Sigma_{\psi_1,\dots,\psi_t}$ is the invariant defined in \cite{ChapmanGilatVishne:2017} for tight sets of quadratic Pfister forms. The following section thus generalizes the results from \cite[Section 5]{ChapmanGilatVishne:2017}.
\end{rem}

\section{From Separable to Inseparable Linkage}

In this section, $F$ is a field of characteristic 2.

\begin{thm}
Let $\psi_1 \in PF_{m+k}$ and $\psi_2 \in PF_{n+k}$ be $k$-linked Pfister forms (that is, separably $k$-linked).
\begin{enumerate}
\item If $\psi_1$ and $\psi_2$ are also inseparably $k$-linked, then their $k$-invariant is trivial, i.e., $\mathcal{F}_{k}(\psi_1 , \psi_2)$ is hyperbolic. 

\item If the $k$-invariant of $\psi_1$ and $\psi_2$ is trivial, and $m=1$ or $n=1$ (that is, at least one of them is a ($k+1$)-fold Pfister form), then the forms are inseparably $k$-linked 
\end{enumerate}
\end{thm}

\begin{proof}
\
\begin{enumerate}

\item Let $\psi_1 = \langle \! \langle a_1, \dots ,a_m \rangle \! \rangle \otimes \pi$ and $\psi_2 = \langle \! \langle b_1, \dots ,b_n \rangle \! \rangle \otimes \pi$. Consider $\mathcal{F}_{k} (\langle \! \langle a_1 \rangle \! \rangle \otimes \pi , \langle \! \langle b_1 \rangle \! \rangle \otimes \pi) = \langle \! \langle a_1 \rangle \! \rangle \otimes \langle \! \langle b_1 \rangle \! \rangle \otimes \pi$. By [CGV, theorem 5.3], since $\langle \! \langle a_1 \rangle \! \rangle \otimes \pi , \langle \! \langle b_1 \rangle \! \rangle \otimes \pi$ are inseparably $k$-linked, $\langle \! \langle a_1 \rangle \! \rangle \otimes \langle \! \langle b_1 \rangle \! \rangle \otimes \pi$ is hyperbolic, but it is also a Pfister factor of $\mathcal{F}_{k}(\psi_1 , \psi_2)$, hence it is isotropic, hence hyperbolic, so it's trivial.

\item Again, let $\psi_1 = \langle \! \langle a_1 , \dots , a_m \rangle \! \rangle \otimes \pi$ and $\psi_2 = \langle \! \langle b_1 , \dots , b_n \rangle \! \rangle \otimes \pi$. Now, denote by $\theta_1$ , $\theta_2$ the unique form (up to isometry) such that: $$\psi_1 = \pi \perp \theta_1$$ $$ \psi_2 = \pi \perp \theta_2 $$
Note that $\dim(\theta_1) = 2^{m+k}-2^{k}$ and $\dim(\theta_2) = 2^{n+k}-2^{k}$. Define $\omega = \langle 1 \rangle \perp \theta_1 \perp \theta_2$. Note that $\dim(\omega) = 2^{m+k}+2^{n+k}-2^{k+1}+1$, and moreover, $\omega$ is a subform of $\mathcal{F}_{k}(\psi_1 , \psi_2)$. Now, without loss of generality, suppose $n=1$. Then $\dim(\omega) = 2^{m+k}+1 > 2^{m+k} = \frac{1}{2}2^{m+k+1} = i_{W}(\mathcal{F}_{k}(\psi_1 , \psi_2))$, where the last equality is true since $\dim(\mathcal{F}_{k}(\psi_1 , \psi_2)) = 2^{m+k+1}$ and it is hyperbolic by assumption. Thus, $\omega$ is isotropic, and hence by Proposition \ref{inslinkedcond}, $\psi_1$ and $\psi_2$ are inseparably $k$-linked, as desired.
\end{enumerate}
\end{proof}

\begin{rem}
Part (1) can be stated for a finite system of $k$-linked Pfister forms, following the exact same argument as in the proof.
\end{rem}

\begin{exmpl}[{Counter-example for (2) with $n,m \geq 2$}]
Let $\psi_1$ ($m+k$)-fold Pfister form and $\psi_2$  ($n+k$)-fold Pfister form such that they are separably $k$-linked, but not inseparably $k$-linked. Then by Proposition \ref{inslinkedcond}, $\omega$ is anisotropic. Now, define $\delta = \mathcal{F}_{k}(\psi_1, \psi_2)$. Note that $\dim(\omega) = 2^{m+k}+2^{n+k}-2^{k+1}+1$, and $\dim(\delta) = 2^{n+m+k}$. Now, notice that when $m,n \geq 2$, $2^{m+k}+2^{n+k}-2^{k+1}+1 < 2^{n+m+k-1}+1$, hence $2\dim(\omega)-2 < \dim(\delta)$, which means that $2\dim(\omega)-1 \leq \dim(\delta)$, and so by the separation theorem \cite[Theorem 26.5]{EKM}, $\omega_{F(\delta)}$ is still anisotropic, hence $\psi_1$ and $\psi_2$ over $F(\delta)$ are not inseparably $k$-linked (Proposition \ref{inslinkedcond}), but their invariant over $F(\delta)$ is trivial.  
\end{exmpl}

\begin{exmpl}[{Counter-example for (2) with more than 2 forms, with at least one of them being ($k+1$)-fold}]
Let $F = \mathbb{F}_{2}^{sep}(x,y)$. Consider the following three forms: $\langle \! \langle \beta, \alpha ] \! ]$ , $\langle \! \langle \alpha, \beta ] \! ]$ , $\langle \! \langle \beta, \alpha \beta ] \! ]$. They are all separably 1-linked, but not inseparably 1-linked (see \cite[Section 5]{Chapman:2018}). Their invariant is trivial because they are pair-wise inseparably 1-linked.
In order to see why they are separably 1-linked, one can repeat the argument in \cite[Example 2.4]{Chapman:2021}, utilizing the fact that this field is $C_2$ to show that they are either separably 1-linked or inseparably 1-linked, and then either explain that since they are not inseparably 1-linked, they must be separably 1-linked, or use \cite{ChapmanFlorenceMcKinnie:2022} to recall that inseparably linkage implies separable linkage anyway.

\end{exmpl}

\section{From $k$- to $(k+1)$-linkage}

\begin{thm}\label{ktkpo}
Let $F$ be a field of any characteristic, with $-1 \in (F^{\times})^{2}$.
Let $\psi_1$ and $\psi_2$ be ($k+1$)-linked Pfister forms. Then their $k$-invariant is trivial, i.e., $\mathcal{F}_{k}(\psi_1 , \psi_2)$ is hyperbolic.
\end{thm}

\begin{proof}
Let $\psi_1 = \langle \! \langle a , b_1 , \dots , b_m \rangle \! \rangle \otimes \pi$ and $\psi_2 = \langle \! \langle a , c_1 , \dots , c_n \rangle \! \rangle \otimes \pi$, where $\pi$ is a $k$-fold Pfister form. Since $-1 \in (F^{\times})^{2}$, $\langle \! \langle a \rangle \! \rangle = \langle \! \langle -a \rangle \! \rangle$, we get that $\mathcal{F}_{k}(\psi_1 , \psi_2 ) = \langle \! \langle a , b_1 , \dots , b_m \rangle \! \rangle \otimes \langle \! \langle -a , c_1 , \dots , c_n \rangle \! \rangle \otimes \pi$. This is clearly an isotropic Pfister form, and thus it is hyperbolic.
\end{proof}
\begin{rem}
For any $\psi$ that is $k$-linked to $\langle \! \langle -1 \rangle \! \rangle ^{m}$ for some positive integer $m$ bigger than $k$, the invariant $\mathcal{F}_{k}(\psi,\langle \! \langle -1 \rangle \! \rangle ^{m})$ is trivial.
\end{rem}

\begin{exmpl}[{Counter-example for Theorem \ref{ktkpo} without assuming $-1 \in (F^{\times})^{2}$}]
A trivial example for that would be $F = \mathbb{R}$ and take the $n$-fold Pfister forms $\psi_1 = \psi_2 = \langle \! \langle -1, \dots, -1 \rangle \! \rangle$. For each $k \in \{1,\dots,n-1\}$, their $k$-invariant is $\mathcal{F}_{k} (\psi_1, \psi_2) = \langle \! \langle \underbrace{-1, \dots, -1}_{2n-k \ \text{times}} \rangle \! \rangle$. This invariant is nontrivial, despite the fact that the forms are also $(k+1)$-linked.
\end{exmpl}

\begin{exmpl}[{Counter-example for the converse of Theorem \ref{ktkpo}}]
Let $\psi_1$ $m$-fold Pfister form and $\psi_2$ $n$-fold Pfister form over a field $F$, $k$-linked but not ($k+1$)-linked. Write $\psi_1 = \varphi_1 \otimes \pi = \mu_1 \perp \pi$ and $\psi_2 = \varphi_2 \otimes \pi = \mu_2 \perp \pi$. By Proposition \ref{seplinkedcond}, $\omega = \mu_1 \perp \mu_2$ is anisotropic. Denote $\delta = \mathcal{F}_{k}(\psi_1, \psi_2)$. Extending scalars to $F(\delta)$, we get that $\delta$ is now hyperbolic, so the invariant is trivial. But, $\dim(\omega) = 2^{n}+2^{m}-2^{k+1} \leq 2^{n+m-k-1} < 2^{n+m-k} = \dim(\delta)$, and by the separation theorem \cite[Theorem 26.5]{EKM}, $\omega$ is anisotropic, and by Proposition \ref{seplinkedcond}, $\psi_1$ and $\psi_2$ are not ($k+1$)-linked (over $F(\delta)$), although their invariant is trivial.
\end{exmpl}

\section*{Acknowledgements}

The authors wish to thank the referee for the careful reading of the manuscript and the helpful comments.

\bibliographystyle{amsplain}
\def\cprime{$'$}
\providecommand{\bysame}{\leavevmode\hbox to3em{\hrulefill}\thinspace}
\providecommand{\MR}{\relax\ifhmode\unskip\space\fi MR }
\providecommand{\MRhref}[2]{%
  \href{http://www.ams.org/mathscinet-getitem?mr=#1}{#2}
}
\providecommand{\href}[2]{#2}

\end{document}